\documentclass{amsart}
\usepackage{amsthm}
\usepackage{amsmath}
\usepackage{amsfonts}
\usepackage{amssymb}
\usepackage{mathrsfs}
\usepackage{mathtools}
\usepackage{tikz-cd}
\usepackage{enumitem}
\usepackage[utf8]{inputenc}
\usepackage[left=30mm,right=30mm]{geometry}
\usepackage{hyperref}
\usepackage{caption}
\usepackage{subcaption}

\theoremstyle{plain}
\newtheorem{thm}{Theorem}[section]
\newtheorem{lem}[thm]{Lemma}
\newtheorem{prop}[thm]{Proposition}
\newtheorem{cor}[thm]{Corollary}

\newtheorem{thmintro}{Theorem}

\theoremstyle{definition}
\newtheorem{defn}[thm]{Definition}
\newtheorem{rem}[thm]{Remark}

\newtheorem{ex}[thm]{Example}

\newcommand{\mc}{\mathcal}

\newcommand{\R}{\mathbb{R}}
\newcommand{\N}{\mathbb{N}}

\newcommand{\mb}{\partial_*}
\newcommand{\ov}{\overline}

\DeclarePairedDelimiter\ceil{\lceil}{\rceil}

\usepackage{xcolor}

\DeclarePairedDelimiter\abs{\lvert}{\rvert}%
\newcommand{\cay}{\mathrm{Cay}(G, \mc S)}
\newcommand{\pres}{\ensuremath{\langle\mc S|\mc R\rangle} }

\title{Small cancellation groups with and without sigma-compact Morse boundary}

\author{Stefanie Zbinden}

\begin{document}

\begin{abstract}
    We provide examples of classical $C'(1/6)$--small-cancellation groups which have non-$\sigma$-compact Morse boundary. These are first known examples of groups with non-$\sigma$-compact Morse boundary. Some $C'(1/6)$--small-cancellation groups  do have $\sigma$-compact Morse boundary, so this property distinguishes quasi-isometry types of small-cancellation groups. In fact, we give a complete description of when Morse boundaries of $C'(1/6)$--groups have $\sigma$-compact Morse boundary. We also provide examples of $C'(1/6)$--groups where all Morse rays are strongly contracting.
\end{abstract}

\maketitle

\section{Introduction}

The Morse boundary is a quasi-isometry invariant which was introduced for CAT(0) groups in \cite{CS:contracting} and generalised for all finitely generated groups in \cite{C:Morse}.

The Morse boundary is neither compact nor metrizable for non-hyperbolic groups \cite{CD:stable, M:CAT0}, but for all known examples so far the Morse boundary is $\sigma$-compact, such as in the cases where it has been fully described \cite{charney2019complete,Z:manifold}, and in all groups (quasi-isometric to a space) where all Morse rays are strongly contracting such as CAT(0) groups and coarsely helly groups, including hierarchically hyperbolic groups \cite{S:CAT0, HHP:injective, SZ:injective}.

We show that, in contrast, the Morse boundaries of $C'(1/6)$--small-cancellation groups exhibit a variety of behaviours. As suggested in \cite{charney2019complete}, we show that there indeed exist $C'(1/6)$--small-cancellation groups with non-$\sigma$-compact Morse boundary. Further we show that not all infinitely presented $C'(1/6)$--groups have non-$\sigma$-compact Morse boundary. In fact for some, being Morse is equivalent to being strongly contracting. We emphasise that, in particular, $\sigma$-compactness can be used to distinguish quasi-isometry types of small cancellation groups. While the theory of Morse boundaries is largely developed in analogy with Gromov boundaries, this is a genuinely new phenomenon unique to the Morse boundary.

\begin{thmintro}\label{thm:main}
    For each of the following properties, there exists an infinitely presented, finitely generated $C'(1/6)$--group $G = \pres$ satisfying that property. 
    \begin{enumerate}
        \item The Morse boundary of $G$ is non-$\sigma$-compact.\label{cond:non-sigma}
        \item Every Morse geodesic in $\cay$ is strongly contracting.
        \item The Morse boundary of $G$ is $\sigma$-compact and there exists a Morse ray in $\cay$ which is not strongly contracting.
    \end{enumerate}
\end{thmintro}

The proof of the above theorem heavily relies on the work of \cite{arzhantseva2017characterizations} showing that being Morse is equivalent to being contracting and the work of \cite{arzhantseva2019negative} where they develop a tool to determine whether rays are Morse (or strongly contracting) in $C'(1/6)$--groups. 

We further give a characterisation that we call increasing partial small-cancellation condition (IPSC) to show when $C'(1/6)$--groups have non-$\sigma$-compact Morse boundary. Roughly speaking, groups satisfy the IPSC, if there exist subwords of longer and longer relators which are a significant fraction of that relator and have very small intersection with other relators.

\begin{thmintro}\label{thm:main2}
The Morse boundary of a finitely generated $C'(1/6)$--group $G = \pres$ is non-$\sigma$-compact if and only if $G$ satisfies the IPSC.
\end{thmintro}

Theorem \ref{thm:main2} is in some sense an upgraded version of Theorem \ref{thm:main} \eqref{cond:non-sigma}. The key to upgrade the proof is a technique used in the proof of Lemma \ref{lemma:ipsc_to_strong} to construct geodesics in $C'(1/6)$--groups that contain desired subwords.

\subsection*{Outline} In Section \ref{sec:prelim}, we recall background on small-cancellation. In Section \ref{sec:non-sigma-compact} we introduce the IPSC and the $C'(1/f)$--small-cancellation condition for functions $f$. The latter is a small-cancellation condition where pieces of longer relators have to be a smaller portion of the relator than those of shorter relators. We show that all $C'(1/f)$--groups have non-$\sigma$-compact Morse boundary. The main idea behind the proof is that in $C'(1/f)$--groups, we can find sequences of geodesics which have very nice contraction properties for balls up to a certain size, but arbitrarily bad contraction properties for larger balls. Lastly, we show that satisfying the IPSC and having non-$\sigma$-compact Morse boundary is equivalent for $C'(1/6)$--groups. A key part of this proof is a strategy developed in the proof of \ref{lem:strong_to_non-sigma}, which allows us to construct geodesics that contain desired subwords. In Section \ref{sec:strongly_contracting} we construct a class of infinitely presented $C'(1/6)$--groups where being Morse is equivalent to being strongly contracting. We ensure this by making sure that the intersection of different relators is large enough. A more detailed outline of the strategy can be found at the beginning of Section \ref{sec:strongly_contracting}. Lastly in Section \ref{sec:intermidate} we adapt the construction of Section \ref{sec:strongly_contracting} to get a group with $\sigma$-compact Morse boundary and where not all Morse rays are strongly contracting. 

\subsection*{Acknowledgements} I want to thank Matthew Cordes and my supervisor Alessandro Sisto for helpful discussions.

\section{Preliminaries}\label{sec:prelim}

\textbf{Notation and Conventions:} For the rest of the paper, unless specified otherwise, $\mc S$ denotes a finite set of formal variables, $\mc S^{-1}$ its formal inverses and $\ov{\mc S}$ the symmetrised set $ \mc S \cup \mc S^{-1}$. A word $w$ over $\mc S$ (respectively $\ov{\mc S}$) is a finite sequence of elements in $\mc S$ (respectively $\ov{\mc S}$). We denote by $w^-$ and $w^+$ the first and last letter of $w$. By abuse of notation, we sometimes allow words to be infinite.

Let $G = \pres$ be  finitely generated group, $X = \cay$ its Cayley graph, $p$ an edge path in $X$ and $v$ a word over $\ov {\mc S}$.
\begin{itemize}
    \item By following $p$, we can read a word $w$ over $\ov {\mc S}$. We say that $p$ is labelled by $w$. We say a word $w'$ is a subword of $p$ if it is a subword of $w$.
    \item For any vertex $x\in X$ there is a unique edge path labelled by $v$ and starting at $x$.
\end{itemize}

\subsection{Small-cancellation}

We will subsequently define most notions needed in our paper. For further background on small-cancellation we refer to \cite{lyndon1977combinatorial}. 

We say a word $w$ over $\ov {\mc S}$ is \emph{cyclically reduced} if it is reduced and all its cyclic shifts are reduced. Given a set $\mc R$ of cyclically reduced words, we denote by $\overline{\mc R}$ the cyclic closure of $\mc R\cup \mc R^{-1}$. If $\mc R = \{w\}$ we sometimes denote $\ov{\mc{R}}$ by $\ov w$.

\begin{defn}[Piece]
    Let $\mc S$ be a finite set and let $\mc R$ be a set of cyclically reduced words over $\ov{\mc S}$. We say that $p$ is a piece if there exists distinct words $r, r'\in \overline{\mc R}$ such that $p$ is a prefix of both $r$ and $r'$. We say that $p$ is a piece of a word $r\in \ov{\mc R}$ if $p$ is a piece and a subword of $r$.
\end{defn}

\begin{defn}[$C'(\lambda)$ condition]
    Let $\lambda >0$ a constant. We say that a set $\mc R$ of cyclically reduced words satisfies the $C'(\lambda)$--small-cancellation condition if for every word $r\in \overline{\mc R}$ and every piece $p$ of $r$ we have $\abs{p}<\lambda\abs{r}$. 
\end{defn}
If $\mc R$ satisfies the $C'(\lambda)$--small-cancellation condition we call the finitely generated group $G = \pres$ a $C'(\lambda)$--group. If $G = \pres$ is a $C'(\lambda)$--group, then the graph $\Gamma$ defined as the disjoint union of cycle graphs labelled by the elements of $\mc R$ is a $Gr'(\lambda)$-labelled graph as defined in \cite{gruber2018infinitely}. We can thus state and use the results of \cite{gruber2018infinitely} and \cite{arzhantseva2019negative} in the less general setting of groups satisfying the $C'(\lambda)$--small-cancellation condition.  

\begin{lem}[Lemma 2.15 of \cite{gruber2018infinitely}]\label{lem:isometric_embedding} Let $G = \pres$ be a $C'(1/6)$--group. Let $r\in \mc R$ be a relator, $\Gamma_0$ a cycle graph labelled by $r$ and let $f\colon \Gamma_0\to \cay$ be a label-preserving graph homomorphism. Then $f$ is an isometric embedding, and its image is convex.
\end{lem}

We call the image of such a label-preserving graph homomorphism an \emph{embedded component}. 

\subsection{Disk diagrams}

\begin{defn}
    A (disk) diagram is a contractible, planar 2-complex. A disk diagram is 
    \begin{itemize}
        \item \emph{simple}, if it is homeomorphic to a disk.
        \item \emph{$\mc S$-labelled} if all edges are labelled by an element of $\ov{S}$.
        \item a diagram \emph{over $\mc R$} if the boundary of any face is labelled by an element of $\ov{\mc R}$
    \end{itemize}
\end{defn}

Let $D$ be a disk diagram and let $\Pi$ be a face of $D$. An \emph{arc} is a maximal subpath of $D$ whose interior vertices all have degree 2. An arc (face) is an \emph{interior arc} (interior face) if it is contained in the interior of $D$ and an \emph{exterior arc} (exterior face) otherwise. Note that an exterior arc is contained in the boundary of $D$. The \emph{interior degree} of a face $\Pi$ is the number of interior arcs in its boundary and the \emph{exterior degree} of a face $\Pi$ is the number of exterior arcs in its boundary.

\begin{defn}[Combinatorial geodesic bigon] A combinatorial
geodesic bigon $(D, \gamma_1, \gamma_2)$ is a simple diagram $D$ whose boundary $\partial D$ is a concatenation of $\gamma_1$ and $\gamma_2$ and such that the following conditions hold.
\begin{enumerate}
    \item Each boundary face whose exterior part is a single arc contained in one of the sides $\gamma_i$ has interior degree at least 4.\label{cond:ext}
    \item The boundary of each interior face consists of at least 7 arcs. \label{cond:int}
\end{enumerate}
\end{defn}

Combinatorial geodesic bigons have been classified in \cite{strebel1990small} as follows.

\begin{figure}
\centering
\begin{minipage}{.5\textwidth}
  \centering
  \includegraphics[width=.9\linewidth]{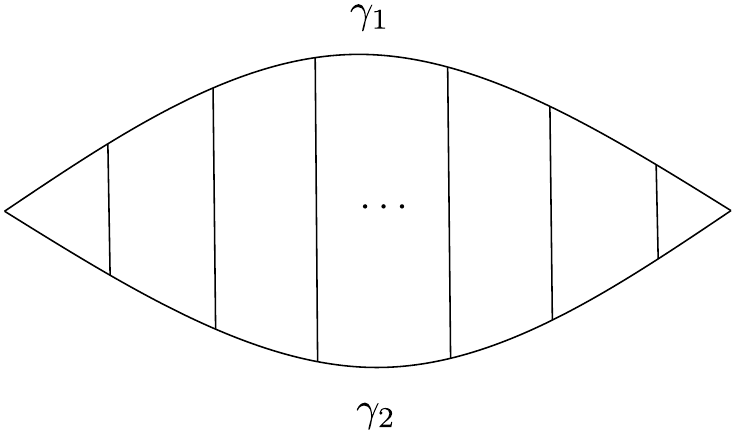}
  \captionof{figure}{Shape $I_1$ of a combinatorial geodesic bigon.}
  \label{fig:i1}
\end{minipage}%
\begin{minipage}{.5\textwidth}
  \centering
  \includegraphics[width=.9\linewidth]{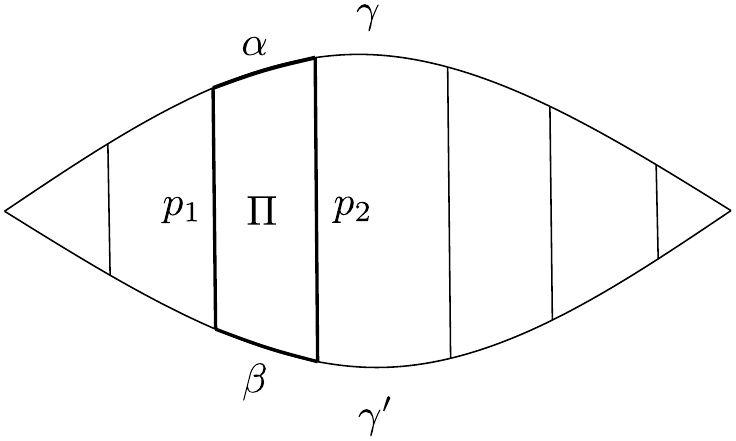}
  \captionof{figure}{Diagram $D$ in the proof of Lemma \ref{lemma:geodesic_condition}.}
  \label{fig:geo_cond}
\end{minipage}
\end{figure}

\begin{lem}[Strebel’s classification,{\cite[Theorem 43]{strebel1990small}}] \label{lemma:strebel} Let $D$ be a combinatorial geodesic bigon. Either $D$ consists of a single face or $D$ has shape $I_1$ as depicted in Figure \ref{fig:i1}, where any but the left and rightmost faces are optional. 
\end{lem}

We can use Strebel's classification to prove the following result, which is our main tool to show that certain paths we care about are in fact geodesics.

\begin{lem}\label{lemma:geodesic_condition}
Let $G = \langle \mc S| \mc R\rangle$ be a $C'(1/6)$--group and let $X = Cay(G, \mc S)$. Let $\gamma$ be a path in $X$ labelled by a reduced word $w$. If for every common subword $u$ of both $w$ and a relator $r$ we have that 
\begin{align}\label{eq:geodesic_condition}
    \abs{u}\leq \frac{\abs{r}}{3},
\end{align}
then $\gamma$ is a geodesic.

\end{lem}

\begin{proof}
Assume that the statement does not hold and let $\gamma$ be the shortest path which satisfies \eqref{eq:geodesic_condition} but is not a geodesic. Let $x$ and $y$ be the start and endpoint of $\gamma$ and let $\gamma'$ be a geodesic, labelled by say $v$, from $y$ to $x$. It is well know (see e.g. \cite[Lemma 2.13]{gruber2015groups}) that there exists an $\ov{\mc S}$--labelled diagram $D$ over $\mc R$ whose boundary is labelled by $wv$ and where every interior arc is a piece. The minimality of $\gamma$ implies that $D$ is simple. The small-cancellation condition ensures that every interior face of $D$ has degree at least 7. Furthermore, let $\Pi$ be a face, labelled by say $r$, whose exterior part is a single arc $\lambda$ contained either in $\gamma$ or $\gamma'$. If $\lambda$ is contained in $\gamma$, $\abs{\lambda}\leq \abs{r}/3$ by \eqref{eq:geodesic_condition}. On the other hand, if $\lambda$ is contained in $\gamma'$, then $\abs{\lambda}\leq \abs{r}/2$ because $\gamma'$ is a geodesic. In either case, the small-cancellation condition ensures that $\Pi$ has interior degree at least 4. Therefore, the diagram $(D, \gamma, \gamma')$ is a combinatorial geodesic bigon. Furthermore, property \eqref{eq:geodesic_condition} and the fact that $\gamma'$ is a geodesic imply that $D$ consists of at least two faces. 

Strebel's classification (Lemma \ref{lemma:strebel}) yields that $D$ has shape $I_1$ as depicted in Figure \ref{fig:i1}. In particular, as depicted in Figure \ref{fig:geo_cond}, the boundary $\partial \Pi$ of any face $\Pi$ of $D$ can be divided into four (possibly) empty paths denoted by $\alpha$, $\beta$, $p_1$ and $p_2$, where $\alpha$ is a subpath of $\gamma$, $\beta$ is a subpath of $\gamma'$ and $p_1, p_2$ are pieces. By \eqref{eq:geodesic_condition} and the small-cancellation condition, $\abs{\alpha}+\abs{p_1}+\abs{p_2}<2\abs{\partial \Pi}/3$ and hence $\abs{\beta} > \abs{\alpha}$. Since this holds for all faces, we have that $\abs{\gamma}< \abs{\gamma'}$, a contradiction to $\gamma'$ being a geodesic. 
\end{proof}

\subsection{Contraction and the Morse boundary} In this section we highlight some consequences of \cite{arzhantseva2017characterizations} and \cite{arzhantseva2019negative}. Forfurhter background on the Morse boundary and contraction we refer to \cite{C:survey} and \cite{arzhantseva2019negative} respectively. 

\begin{defn}[Intersection function]
Let $G = \pres$ be $C'(1/6)$--group. Let $\gamma$ be an edge path in $\cay$. The intersection function of $\gamma$ is the function $\rho : \N\to \R_+$ defined by 
\begin{align*}
    \rho(t) = \max_{\substack{\abs{r}\leq t\\r\in\ov{\mc R}}}\left\{\abs{w}\Big\vert \text{$w$ is a subword of $r$ and $\gamma$}\right\}.
\end{align*}
\end{defn}
\begin{rem}
    In light of Lemma \ref{lem:isometric_embedding}, the intersection function $\rho$ of a geodesic $\gamma$ is equal to the function $\rho'$, defined via
    \begin{align*}
    \rho'(t) = \max_{\abs{\Gamma_0}\leq t}\left\{\abs{\Gamma_0 \cap \gamma}\right\},
    \end{align*}
    where $\Gamma_0$ ranges over all embedded components.
\end{rem}

The following lemma is a combination of \cite[Theorem 1.4]{arzhantseva2017characterizations} and \cite[Corollary 4.14, Theorem 4.1]{arzhantseva2019negative}. It shows the relation between contraction, Morseness and the intersection function.

\begin{lem}[Theorem 4.1, Corollary 4.14 of \cite{arzhantseva2019negative}]\label{lem:4.14} Let $G = \pres$ be a $C'(1/6)$--group. Let $\alpha$ be a geodesic in $X =\cay$ and let $\rho$ be its intersection function. Then,
\begin{enumerate}[label=(\roman*)]
    \item The geodesic $\alpha$ is Morse if and only if $\rho$ is sublinear.
    \item The geodesic $\alpha$ is strongly contracting if and only if $\rho$ is bounded.
    \item If $\rho$ is sublinear, $\alpha$ is $M$-Morse for some Morse gauge $M$ only depending on $\rho$.
    \item If $\alpha$ is $M$-Morse, then $\rho\leq \rho'$ for some sublinear function $\rho'$ only depending on $M$. 
\end{enumerate}
\end{lem}

\begin{proof}
    Corollary 4.14 of \cite{arzhantseva2019negative} states that (i) and (ii) hold. To get (iii) observe the following; Since $\alpha$ is a geodesic, it intersects every embedded component $\Gamma_0$, with say $\abs{\Gamma_0} = \ell$, in a subsegment of length at most $\ell/2$. Thus for any $x\in \Gamma_0$ with $d(x, \Gamma_0\cap \alpha) = r \leq \ell/8$, the projection of $B_{\Gamma_0}(x, r)$ in $\Gamma_0$ onto $\Gamma_0\cap \alpha$ projects to a single point (namely one of the two endpoints of $\Gamma_0\cap \alpha$). Therefore, $\alpha\cap \Gamma_0$ is $(r ,\rho')$-contracting in $\Gamma_0$ for $\rho'(r) = \rho(8r)$. Theorem 4.1 of \cite{arzhantseva2019negative} then implies that $\alpha$ is $(r, \rho'')$-contracting in $X$ for some $\rho''$ only depending on $\alpha$. By Theorem 1.4 of \cite{arzhantseva2017characterizations}, $\alpha$ is $M$-Morse for some Morse gauge $M$ only depending on $\rho''$ (and hence only depending on $\rho$).

    Lastly, we prove (iv): By Theorem 1.4 of \cite{arzhantseva2017characterizations}, $\alpha$ is $(r, \rho')$-contracting for some sublinear function $\rho'$ only depending on $M$. We may assume the $\rho'$ is increasing. Let $\Gamma_0$ be an embedded component which intersects $\alpha$ and let $x\in \Gamma_0$ be a point maximising $l = d(x, \Gamma_0\cap \alpha)$. The projection of $B_{\Gamma_0}(x, l)$ onto $\Gamma_0\cap \alpha$ (and, by Lemma \ref{lem:isometric_embedding}, the projection of $B_X(x, l)$ onto $\alpha$) has diameter $\abs{\Gamma_0\cap \alpha}$. Hence $\rho'(\abs{\Gamma_0})\geq \rho'(l) \geq \abs{\Gamma_0\cap\alpha}$. Since this is true for all embedded components $\Gamma_0$ we have that $\rho\leq \rho'$.
\end{proof}

\begin{defn}[Contraction Exhaustion]
    Let $G = \pres$ be a $C'(1/6)$--group. A contraction exhaustion of $X =\cay$ is a sequence of sublinear functions $\rho_i : \N \to \R_+$ such that $\rho_i\leq \rho_{i+1}$ for all $i$ and the following property holds. The intersection function $\rho$ of any Morse geodesic ray $\gamma$ in $X$ satisfies $\rho\leq \rho_i$ for some $i$.
\end{defn}

\begin{lem}\label{lem:contraction_exhaustion}
Let $G = \pres$ be a $C'(1/6)$--group. A contraction exhaustion of $\cay$ exists if and only if $\mb G$ is $\sigma$-compact.
\end{lem}
\begin{proof}
    The Morse boundary $\mb G$ is $\sigma$-compact if and only if there exists a sequence of Morse gauges $M_1\leq M_2\leq \ldots$ such that any Morse geodesic ray in $\cay$ is $M_i$-Morse for some $i$. Lemma \ref{lem:4.14} iii) and iv) conclude the proof. 
\end{proof}

\section{Criterion for non sigma-compact Morse boundaries}\label{sec:non-sigma-compact}
In this section we introduce the increasing partial small-cancellation condition (IPSC). We then prove Theorem \ref{prop:non-sigma-compact-condition}, which states that $C'(1/6)$--groups have non-sigma compact Morse boundary if and only if they satisfy the IPSC. We conclude the section by giving examples of $C'(1/6)$--groups that satisfy IPSC and hence have non-$\sigma$-compact Morse boundary. 

Since the proof of Theorem \ref{prop:non-sigma-compact-condition} is quite technical, we start the section by, in Subsection \ref{section:c1f}, proving Lemma \ref{lemma:caf_case}, which is a weaker result than Theorem \ref{prop:non-sigma-compact-condition}, but illustrates the ideas of the proof without having to go into many of the technical details.

\subsection{$C'(1/f)$--groups do not have $\sigma$-compact Morse boundary}\label{section:c1f}

We introduce $C'(1/f)$--groups and show that such groups do not have $\sigma$-compact Morse boundary. The proof illustrates the main ideas of the proof of Lemma \ref{lem:strong_to_non-sigma}.

\begin{defn}\label{def:viable}
    We say that a weakly increasing function $f: \N \to \R_+$ is viable if $f(n)\geq 6$ for all $n$ and $\lim_{n\to \infty} f(n) = \infty$. 
\end{defn}

Observe that a consequence of a function $f$ being viable is that the function $\rho(n) = n/f(n)$ (or equivalently the function $\rho'(n) =\max\{\rho'(n-1), n/f(n)\}$) is sublinear. 

\begin{defn}\label{def:c1f} Let $f$ be a viable function. We say that a finitely generated group $G = \pres$ is a $C'(1/f)$--group if $\mc R$ is infinite and for every piece $p$ of a relator $r\in \ov{\mc R}$ we have that $\abs{p}<\abs{r} / f(\abs{r})$. 
\end{defn}

In other words, in $C'(1/f)$--groups, pieces of large relators have to be smaller fractions of their relators than pieces of smaller relators. Requiring that viable functions satisfy $f(n)\geq 6$ ensures that all $C'(1/f)$--groups are $C'(1/6)$--groups.
 
\begin{lem}\label{lemma:caf_case}
    Let $f$ be a viable function, let $G = \pres$ be a $C'(1/f)$--group and let $X = \cay$ be its Cayley graph. The Morse boundary $\mb G$ is not $\sigma$-compact. 
\end{lem}

\smallskip

To do so, we assume that $\mb G$ is $\sigma$-compact and show that this leads to a contradiction.  Note, we can and will assume that $\mc R = \ov {\mc R}$. By Lemma \ref{lem:contraction_exhaustion}, there exists a contraction exhaustion $(\rho_i)_{i\in\N}$ of $X$. To get to a contradiction we construct a geodesic ray $\gamma$ which is Morse but whose intersection function $\rho$ satisfies $\rho\not \leq \rho_i$ for all $i\in \N$.

\smallskip

\textbf{Construction of the geodesic $\gamma$}: Let $N\geq 18$ be an integer. For every integer $i\geq N$ let $k_i\geq n_i$ be integers such that $f(n_i)>  4i$ and $\rho_i(t)< t/(2i)$ for all $t\geq k_i$ (such an integer $k_i$ exists since $\rho_i$ is sublinear). Let $r_i = x_iy_i\in \mc R$ be a relator such that,
\begin{enumerate}
    \item $\ov{r_i}\neq \ov{r_{i-1}}$,
    \item $\abs{r_i}\geq k_i$,
    \item $\abs{r_i}/(2i)\leq \abs{x_i}\leq \abs{r_i}/i$,
    \item $x_{i-1}x_i$ is reduced. \label{it4}
\end{enumerate}
Observe that such a relator $r_i$ always exists; Properties (1)-(3) can be satisfied by choosing any large enough relator $r$, while Property (4) is satisfied by at least one relator in $\ov{r}$. Observe that $\abs{x_i}\geq \abs{r_i}/2i\geq \abs{r_i} /f(\abs{r_i})$ and hence $x_i$ is not a piece.

\smallskip

Define $\gamma$ as the as ray starting at the identity labelled by the word $\prod_{i=N}^\infty x_i$. It remains to show that $\gamma$ is a geodesic, $\gamma$ is Morse and its intersection function 
\begin{align*}
    \rho(t) = \max_{\substack{\abs{r}\leq t\\r\in\ov{\mc R}}}\left\{\abs{w}\Big\vert \text{$w$ is a subword of $r$ and $\gamma$}\right\}
\end{align*} 
satisfies $\rho\not\leq \rho_i$ for all $i$. By Lemma \ref{lemma:geodesic_condition} the path $\gamma$ is a geodesic if $\rho(t)\leq t/3$ for all $t$. Further, Lemma \ref{lem:4.14} states that $\gamma$ is Morse if $\rho$ is sublinear. We therefore proceed to bound $\rho$.

\smallskip

\textbf{Bounding the intersection function:} Let $w$ be a common subword of $\gamma$ and a relator $r\in \mc R$. That is, $w = u_ix_{i+1}\ldots x_{j-1}u_j$ for some $i < j$ and (possibly empty) subwords $u_i$ of $x_i$ and $u_j$ of $x_j$. By construction, $x_{i+1}$ is not a piece, so either $j = i+1$ or $\ov r = \ov{r_{i+1}}$. 

    \textbf{Case 1:} $\ov r = \ov {r_{i+1}}$ (and $j\geq i+2$).  By construction, $\ov {r_{i+1}}\neq \ov{r_{i+2}}, \ov {r_i}$. Thus $j = i+2$ and both $u_i$ and $u_j$ are pieces. Hence 
    \begin{align*}
        \abs{w} = \abs{u_i} + \abs{u_j} + \abs{x_{i+1}} < \frac{2\abs{r_{i+1}}}{f(\abs{r_{i+1}})} + \frac{\abs{r_{i+1}}}{i+1},
    \end{align*}
    In particular, $\abs{w}<\abs{r_{i+1}}/3 = \abs{r}/3$.

    \textbf{Case 2:} $j = i+1$. If $\ov r \neq \ov{r_i}$ , then $u_i$ is a piece of $r$ and hence $\abs{u_i}<\abs{r}/f(\abs{r}) \leq \abs{r}/6$. If $\ov r = \ov{r_i}$, then $\abs{x_i}\leq \abs{r}/i$. Thus in either case $\abs{u_i} \leq \abs{r}/6$. Doing the same for $u_j$ one can show that $\abs{w}<\abs{r}/3$.

    So indeed $\rho(t)\leq t/3$ for all $t$. Furthermore, summarising the two cases, we get that 
    \begin{align*}
        \abs{w}\leq\begin{cases}
            \frac{2\abs{r_i}}{f(\abs{r_i})} + \frac{2\abs{r_i}}{i} &\text{if $\ov r = \ov r_i$ for some $i\geq N$,}\\
            \frac{2\abs{r}}{f(\abs{r})} &\text{otherwise.}
        \end{cases}
    \end{align*}
    Since $f$ is increasing and unbounded, $\rho$ is indeed sublinear. It remains to show that $\rho\not\leq \rho_i$ for all $i$ or, equivalently, that $\rho\not \leq \rho_i$ for all $i\geq N$. Let $i\geq N$. Recall that $\rho_i(t)< t/(2i)$ for all $t\geq k_i$, so in particular $\rho_i(\abs{r_i})<\abs{r_i} /(2i)$. On the other hand, $\rho(\abs{r_i})\geq \abs{x_i} \geq \abs{r_i}/(2i)$. So indeed $\rho\not \leq \rho_i$.
    
\subsection{The general case}
We first introduce the IPSC and strong IPSC. We then prove Theorem \ref{prop:non-sigma-compact-condition}, the main result of this Section. 

\begin{defn}\label{def:loacl:c1f}
    Let $f$ be a viable function, let $x$ be a reduced word and let $\mc R$ be a set of cyclically reduced words. We say that the pair $(x, \mc R)$ satisfies the $C'(1/f)$--small-cancellation condition if every common subword $p$ of $x$ and a relator $r\in \ov {\mc R}$ satisfies $\abs{p}< \abs{r}/f(\abs{r})$.
\end{defn}

The following definition quantifies having sufficiently long subwords $w$ of longer and longer relators that satisfy the $(w, \mc R)$--small-cancellation condition.

\begin{defn}[IPSC]\label{def:IPSC}
     Let $G = \pres$ be a finitely generated group. We say that $G$ satisfies the \emph{increasing partial small-cancellation condition (IPSC)} if for every sequence $(n_i)_{i\in \N}$ of positive integers, there exists a viable function $f$ such that the following holds; For all $K\geq 0$ there exists $i\geq K$ and a relator $r = xy\in \ov{\mc R}$ satisfying: 
    \begin{enumerate}[label= \roman*)]
        \item $\abs{r}\geq n_i$,
        \item $\abs{x}\geq \abs{r}/i$,
        \item the pair $(x, \mc R)$ satisfies the $C'(1/f)$--small-cancellation condition.
    \end{enumerate}
\end{defn}

\begin{thm}\label{prop:non-sigma-compact-condition}
    Let $G = \pres$ be a $C'(1/6)$--group. The Morse boundary of $G$ is non-$\sigma$-compact if and only if $G$ satisfies the IPSC.
\end{thm}

Theorem \ref{prop:non-sigma-compact-condition} is a direct consequence of Lemma \ref{lemma:non_sigma_to_IPSC},\ref{lemma:ipsc_to_strong} and \ref{lem:strong_to_non-sigma}, which are stated and proven below.  

\begin{lem}\label{lemma:non_sigma_to_IPSC}
    Let $G =\pres$ be a $C'(1/6)$--group. If the Morse boundary of $G$ is not $\sigma$-compact, then $G$ satisfies the IPSC.
\end{lem}

\begin{proof}
    Assume that $\mb G$ is not $\sigma$ compact. Let $(n_i)_{i\in \N}$ be a sequence of positive integers. We may assume that $n_i < n_{i+1}$ for all $i$. 
    Define the function $\rho : \N \to \R_+$ iteratively via $\rho(1) = 1$ and
    \begin{align*}
        \rho(t) = \begin{cases}
            t &\text{ if $t< n_1$,}\\
           \max\left\{ t / i, \rho(t-1)\right\} &  \text{if $n_i\leq t < n_{i+1}$}.
        \end{cases}
    \end{align*}
    
    Furthermore for all $k\in \N$ define the function $\rho_k : \N \to \R_+$ via $\rho_k(t) = \max\{\rho(t), k\}$. Observe that $\rho$ and all $\rho_k$ are sublinear. 
    
    Since $\mb G$ is not $\sigma$-compact, there exists a Morse geodesic ray $\gamma$ whose intersection function $\rho^*$ is sublinear but $\rho^*\not\leq \rho_k$ for all $k$. Define the functions $f, f': \N \to \R_+$ via
    \begin{align*}
        f'(t) =  \min_{t'\geq t}\left\{\frac{t'}{\rho^*(t')}\right\} -1,\quad
         f(t) = \max \left\{ 6, f'(t)\right\}.
    \end{align*}
    Observe that $f$ is viable. Let $K_0$ be the smallest integer such that $f'(K_0)\geq 6$. Let $K\geq K_0$. Since $\rho^*\not\leq\rho_{n_K}$, there exists $t_0\geq 1$ such that $\rho^*(t_0) > \rho_{n_K}(t_0)$. In particular, there exists a relator $r\in \ov{ \mc R}$ of length $m\leq t_0$ and a common subword $x$ of $r$ and $\gamma$ such that $ \abs{x}> \rho_{n_K}(t_0)$. We may assume that $x$ is a prefix of $r$. Let $j$ be such that $n_j\leq m < n_{j+1}$. It remains to show that $j\geq K$ and (ii), (iii) from the definition of IPSC are satisfied for $r = xy$. 
    
    Since $\rho_{n_K}(t_0)\geq n_K$, we have indeed that $j\geq K$. Further, $\rho_{n_K}(m) \geq m/j$ and hence $\abs{x}\geq \abs{r}/j$.      
    Lastly,we prove (iii). Let $r'$ be a relator with $\abs{r'}\geq K_1$ and let $w$ be a common subword of $r'$ and $x$. Since $x$ is a subword of $\gamma$, we have that $\abs{w}\leq \rho^*(\abs{r'}) < \frac{\abs{r'}}{f(\abs{r'})}$. On the other hand, let $r'$ be a relator with $\abs{r'}<K_1$ and let $w$ be a common subword of $r'$ and $x$. Since $\abs{r}\geq K_1$, $w$ is a piece and hence $\abs{w}< \abs{r'}/6 = \abs{r'}/f(\abs{r'})$. So indeed, the pair $(x, \mc R)$ satisfies the $C'(1/f)$--small-cancellation condition. Hence, $G$ satisfies the IPSC, which concludes the proof.
\end{proof}

Before we can state Lemma \ref{lemma:ipsc_to_strong} and \ref{lem:strong_to_non-sigma}, we need to introduce the strong IPSC. 

\begin{defn} Let $w $ be a reduced word over $\ov{\mc S}$. We can write $w = s_1^{k_1}\ldots s_n^{k_n}$, where $s_i\in\ov{\mc{S}}$, $k_i>0$ and $s_{i+1}\neq \{s_i, s_i^{-1}\}$. We say that the word $s_2^{k_2}\ldots s_{n-1}^{k_{n-1}}$ is the \emph{interior} of $w$. In particular, if $n\leq 2$, the interior of $w$ is empty.
\end{defn}

The strong IPSC is similar to the IPSC, but condition ii) is replaced with the stronger version ii*), which allows more control on the interior of $x$. This control will be crucial in the proof of Lemma \ref{lem:strong_to_non-sigma}.

\begin{defn}[Strong IPSC]
     Let $G = \pres$ be a finitely generated group. We say that $G$ satisfies the \emph{strong IPSC} if for every sequence $(n_i)_i$ of integers, there exists a viable function $f$ such that the following holds; For all $K\geq 0$ there exists $i\geq K$ and a relator $r = x y\in \ov{\mc R}$ satisfying: 
    \begin{enumerate}[label= \roman*)]
        \item $\abs{r}\geq n_i$,
        \item [ii*)] $\abs{x'}\geq \abs{r}/i$, where $x'$ denotes the interior of $x$,
        \item [iii)] the pair $(x, \mc R)$ satisfies the $C'(1/f)$--small-cancellation condition.
    \end{enumerate}
\end{defn}

\begin{lem}\label{lemma:ipsc_to_strong}
    Let $G = \pres$ be a $C'(1/6)$--group satisfying the IPSC. Then $G$ satisfies the strong IPSC.
\end{lem}

\begin{proof}
For every generator $s\in \ov{\mc S}$, define the function $\rho_s : \N\to \R_+$ via
\begin{align*}
    \rho_s(t) = \max_{\substack{r\in \ov{\mc R} \\ \abs{r}\leq t}}\{k |\text{$s^k$ is a subword of $r$}\}.
\end{align*}
Observe that the function $\rho_s$ need not be sublinear. Let $\mc S_1$ be the set of generators $s\in\ov{\mc S}$ for which $\rho_s$ is sublinear, and let $\mc S_2 $ be the set of generators $s\in\ov{\mc S}$ for which $\rho_s$ is not sublinear. Since $\mc S$ is finite, the function $\rho : \N\to R_+$ defined via $\rho(t)= \max_{s\in \mc S_1}\{\rho_s(t)\}$ is sublinear.

Let $(n_i)_{i\in \N}$ be a sequence of positive integers. For all $i\geq 1$ choose an integer $n'_i$ such that $n_i '\geq n_{3i}$, $n_i'\geq 3i^2$ and such that $\rho(t)\leq \frac{t}{3i}$ for all $t\geq n_i'$. Since $\rho$ is sublinear we can always find such an integer $n_i'$. 

Next we use that $G$ satisfies the IPSC. Let $f$ be the viable function corresponding to the sequence $(n_i')_i$. Note that $f$ is unbounded. Let $s\in \mc S_2$. Since $\rho_s$ is not sublinear, there exists a constant $D_s$ such that $\rho_s(D_s)> D_s/f(D_s)$. In particular, if $(s^k, \mc R)$ satisfies the $C'(1/f)$--small-cancellation condition, for some $s\in \mc S_2$, then $k < D_s$. Let $D = \max_{s\in \mc S_2} \{D_s\}$. 

Let $K\geq D$. Since $\pres$ satisfies the IPSC, there exists an integer $i\geq K$ and a relator $r\in \ov{\mc R}$ with $r= xy$ and such that 
\begin{enumerate}
    \item $\abs{r}\geq n_i'\geq n_{3i}$
    \item $\abs{x}\geq \abs{r}/i$
    \item the pair $(x, \mc R)$ satisfies the $C'(1/f)$--small-cancellation condition. 
\end{enumerate}
Observe the following; to prove that $G$ satisfies the strong IPSC, it now suffices to show that $\abs{x'}\geq \abs{r}/(3i)$, where $x'$ is the interior of $x$. Write $x = s_s^kx's_e^l$ for some $s_s, s_e\in \ov{\mc S}$. If $s_e\in \mc S_1$, then $k\leq\rho(\abs{r})$ by the definition of $\rho$. By the definition of $n_i'$, we have additionally $\rho(\abs{r})\leq\abs{r}/(3i)$. If $s_e\in \mc S_2$, we have $k\leq D$ as explained above. Since $i\geq D$ and $n_i'\geq 3i^2$ we have in both cases $k\leq \abs{r}/(3i)$. The same holds for $l$. Thus $\abs{x'}\geq \abs{r}/(3i)$, implying that $\pres$ indeed satisfies the strong IPSC.
\end{proof} 

\begin{lem}\label{lem:strong_to_non-sigma}
    Let $G =\pres$ be a $C'(1/6)$--group which satisfies the strong IPSC. Then the Morse boundary of $G$ is not $\sigma$-compact.
\end{lem}

\begin{proof}
    We prove this by contradiction. Assume that $\mb G$ is $\sigma$-compact and 
    Let $(\rho_i)_{i\in\N}$ be a contraction exhaustion of $X = \cay$. As in the proof of Lemma \ref{lemma:caf_case}, we construct a Morse geodesic $\gamma$ whose intersection function $\rho$ satisfies $\rho\not\leq \rho_i$ for all $i$, contradicting the assumption that $\mb G$ is $\sigma$-compact. However, before we can start with the construction of $\gamma$, we need to do some technical setup.
    
    \textbf{Constructing the $x_i$.} Choose any relator $r_0\in \ov{\mc{R}}$ with $\abs{r_0}\geq 42$. For $i\geq 1$ choose $n_i\geq 6i\abs{r_0}$ such that $\rho_i(t)< t/i$ for all $t\geq n_i$ (such an integer always exists since $\rho_i$ is sublinear). Since $\pres$ satisfies the strong IPSC there exists a viable function $f$ and a sequence of relators $r_i = b_i^{l_i}x_ia_i^{j_i}y_i$ and a sequence of indices $k_i\geq i$ such that: 
    \begin{enumerate}
        \item $a_i, b_i\in \ov{\mc{S}}$, $l_1, j_i>0$, and $x_i$ is the interior of $b_i^{l_i}x_ia_i^{j_i}$,
        \item $\abs{r_i}\geq n_{k_i}$,
        \item $\abs{x_i}\geq \abs{r}/k_i$,
        \item the pair $(b_i^{l_i}x_ia_i^{j_i}, \mc R)$ satisfies the $C'(1/f)$--small-cancellation condition. 
    \end{enumerate}

    Unlike in the proof of Lemma \ref{lemma:caf_case}, we cannot guarantee that the words $x_i$ are not pieces or that $x_ix_{i+1}$ is reduced. Unfortunately, in Lemma \ref{lemma:caf_case}, this was a crucial step in showing that $\gamma$ is a Morse geodesic. Thus it is not enough to define $\gamma$ as the path labelled by the word $\Pi_{i\geq 1}^\infty x_i$. Instead, we have to do some surgery on the words $x_i$.

    \textbf{Performing surgery on the $x_i$.} Recall that $w^-$ ($w^+$) denotes the first (last) letter of a word $w$. Let $q$ be a subword of the relator $r_0$ such that $\abs{q}\leq \abs{r_0}/3$ and $\abs{q}-6\geq \abs{r_0}/6$. With this, $12\abs{q}\leq \abs{x_i}$ for all $i$. Write $q = s_1s_2s_3q's_4s_5s_6$ for some letters $s_j\in\ov{\mc S}$ for $1\leq j \leq 6$. Observe that $\abs{q'}\geq \abs{r_0}/6$ and hence $q'$ is not a piece. In the following we show that for each $i\geq 1$ we can choose words $u_i, v_i, w_i, z_i$ such that: 
    \begin{enumerate}
        \item for $i\geq 2$, $q_i = z_{i-1}q'u_i$ is a subword of $q$,
        \item for $i\geq 1$, $y_i = v_ix_iw_i$ is a subword of $b_i^{l_i}x_ia_i^{j_i}$,
        \item the word $w = \prod_{i=1}^\infty y_iq_{i+1}$ is reduced,
        \item for $i\geq 2$, neither $q_iy_i^{-}$ nor $y_{i-1}^+q_i$ is a subword of any relator.\label{4}
    \end{enumerate}

    We show how to choose the words $u_i$ and $v_i$, the words $w_i$ and $z_i$ can be chosen analogously. In the following, $\lambda$ denotes the empty word.

    \textbf{Case 1:} $s_4 = s_5$. 
    \begin{align*}
        u_i = s_4, \qquad v_i = \begin{cases}
            a^{j_i} &\text{if $a\not\in \{s_4, s_4^{-1}\}$}\\
            \lambda &\text{otherwise.}
        \end{cases}
    \end{align*}

    \textbf{Case 2:} $s_4\neq s_5 = s_6$.

    \begin{align*}
        u_i = s_4s_5, \qquad v_i = \begin{cases}
            a^{j_i} &\text{if $a\not\in \{s_5, s_5^{-1}\}$}\\
            \lambda &\text{otherwise.}
        \end{cases}
    \end{align*}
    
    \textbf{Case 3:} $s_4\neq s_5 \neq s_6$
    \begin{align*}
        u_i = \begin{cases}
            s_4 & \text{if $a\not\in\{s_4^{-1}, s_5\}$}\\
            s_4s_5 & \text{if $a\in\{s_4^{-1}, s_5\}$ but $a\not\in \{s_5^{-1}, s_6\}$}\\
            s_4 & \text{if $s_4 = s_6^{-1}$ and $x_i'^- = s_5^{-1}$}\\
            s_4s_5 & \text{if $s_4 = s_6^{-1}$ and $x_i'^-\neq s_5^{-1}$.}
        \end{cases}, \qquad v_i = \begin{cases}
            a^{j_i} &\text{if $s_4\neq s_6^{-1}$}\\
            \lambda &\text{otherwise.}
        \end{cases}
    \end{align*}
     Properties $(1)-(3)$ follow immediately from the choice. Furthermore, the choice assures that if $u_i = s_4$, then $y_i^-\neq s_5$ and if $u_i = s_4s_5$, then $y_i^-\neq s_6$. Since $q'$ (and hence $q_i$) is not a piece, this implies Property (4).

     \textbf{Constructing $\gamma$.} Let $\gamma$ be the path starting at the identity and labelled by $w = \prod_{i=1}^\infty y_iq_{i+1}$ and let $\rho$ be its intersection function. If $\rho(t)\leq t/3$ for all $t$, then $\gamma$ is a geodesic by Lemma \ref{lemma:geodesic_condition}. If in addition, $\rho$ is sublinear, then $\gamma$ is Morse by Lemma \ref{lem:4.14}. So next, we show that these two conditions indeed hold.

     Let $w$ be a common subword of a relator $r$ and $\gamma$. By \eqref{4}, we have one of the following, $w = uy_iv$, $w= uz$ or $w = zv$, where $u$ is a suffix of $q_{i-1}$, $z$ a subword of $y_i$ and $v$ a prefix of $q_i$.

     \textbf{Case 1:} $\ov{r} = \ov{r}_0$. We have that $\abs{r_0} < \abs{y_i}$ and hence $w = uz$ (or $w = zv$). If $z$ is empty, then by the choice of $q$, $\abs{w}\leq \abs{r_0}/3$. If $z$ is not empty, then $u$ (respectively $v$) is a piece by \eqref{4}. Since $\ov{r}_0\neq \ov{r}_i$, the word $z$ is a piece as well and hence $\abs{w}\leq \abs{r_0}/3$.
     
     \textbf{Case 2:} $\ov{r}  \neq \ov{r}_0$. Recall that the pair $(y_i, \mc R)$ satisfies the $C'(1/f)$--small-cancellation condition and $\abs{y_i} > 12\abs{q}$. Thus if $w = uy_iv$, then $\abs{y_i}<\abs{r}/f (\abs{r})$ and hence $\abs{w}< \abs{r}/f(\abs{r}) +2\abs{q}\leq \abs{r}/3$. On the other hand, if $w = uz$, then $u$ is a piece and thus $\abs{w} = \abs{u}+\abs{z} < \abs{r}/6 + \abs{r}/f(\abs{r}) < \abs{r}/3$. Also, $\abs{w}< \abs{r}/f(\abs{r}) + \abs{q}$. We can prove the same if $w = zv$.

    In any case, $\abs{w}\leq \abs{r}/3$, and $\abs{w}\leq \abs{r}/f(\abs{r}) + 2\abs{q}$. The former shows that indeed $\rho(t)\leq t/3$ for all $t$. Since $f$ is unbounded, the latter shows that $\rho$ is indeed sublinear. 
    
    To conclude the proof, it suffices to show that $\rho\not \leq \rho_i$ for all $i$. By construction of the sequence $(n_i)_i$ we have that $\rho_i(t)<t/i$ for all $t\geq n_i$. In particular, $\rho_i(\abs{r_i})<\abs{r_i}/k_i$. Since $y_i$ is a subword of $\gamma$ we have that $\rho(\abs{r_i})\geq \abs{y_i}\geq \abs{r_i}/k_i$. Thus indeed $\rho\not\leq \rho_i$.
\end{proof}

\subsection{Examples of groups with non-$\sigma$-compact Morse boundary}

We reprove Lemma \ref{lemma:caf_case} as a Corollary of Theorem \ref{prop:non-sigma-compact-condition}. Then we show that for at least some viable functions, the class of $C'(1/f)$--groups is not empty. In fact, it contains some of the standard examples of small-cancellation groups.

\begin{cor}\label{cor:c1f}
    Let $f$ be a viable function and let $G = \pres$ be a $C'(1/f)$--group. Then the Morse boundary $\mb G$ is not $\sigma$-compact.
\end{cor}
\begin{proof}
    By Theorem \ref{prop:non-sigma-compact-condition}, it is enough to show that $G$ satisfies the IPSC. Let $(n_k)_k$ be a sequence of integers. We can choose a sequence $(r_k)_k$ of relators such that $\abs{r_k}> \max\{k, n_k, \abs{r_{k-1}}\}$ for all $k$. Define the function $f': \N\to \R_+$ via 
    \begin{align*}
        f'(n) = \begin{cases}
            f(n)& \text{if $n< \abs{r_{12}}$}\\
            \min\{f(n), \frac{k}{2}\} &\text{if $\abs{r_k}\leq n <\abs{r_{k+1}}$ for some $k\geq 12$.}
        \end{cases}
    \end{align*}
    Let $k\geq 12$ and let $x$ be a prefix of a relator $r_k$ satisfying $2\abs{r_k}/k> \abs{x}\geq \abs{r_k}/k$. The pair $(x, \mc R)$ satisfies the $C'(1/f')$--small-cancellation condition. Thus $G$ indeed satisfies the IPSC. 
\end{proof}

\begin{ex}\label{ex:cf}
    Let $\mc S  = \{x, y\}$. For $i\geq 14$ define 
    \begin{align*}
        r_i = \prod_{j=0}^{2i} xy^{i^2+j}.
    \end{align*}
    Note that $\abs{r_i}\geq i^3$. Let $\mc R = \{r_i\}_{i\geq 14}$ and let $f: \N\to \R_+$ be defined as follows: 
    \begin{align*}
        f(i) = \begin{cases}
            6 & \text{if $i < 14$ }\\
            \ceil{\frac{i^3}{2(i+1)^2}}& \text{otherwise.}
        \end{cases}
    \end{align*}
    We show below that the group $G = \pres$ is a $C'(1/f)$--group. As a consequence, its Morse boundary is non-$\sigma$-compact by Corollary \ref{cor:c1f} and its existence implies Theorem \ref{thm:main} (1).
    \smallskip
    Any piece $p$ of a relator $r\in \ov{r_i}$ has the form $y^k$, $y^{-k}$, $y^kxy^{k'}$  or $y^{-k}x^{-1}y^{-k'}$ for some $0\leq k, k'\leq (i+1)^2-1$. Hence $\abs{p}< 2(i+1)^2\leq \abs{r}/f(\abs{r})$. Thus $G = \pres$ is indeed a $C'(1/f)$--group.
\end{ex}

\section{Strongly Contracting Morse boundary}\label{sec:strongly_contracting}

In this section, we construct a class of examples of $C'(1/6)$--groups $G = \pres$ where $\abs{\mc R}$ is infinite and every geodesic Morse ray in $X = \cay$ is strongly contracting. Every geodesic Morse ray being strongly contracting implies that the Morse boundary of $G$ is $\sigma$-compact and hence the examples constructed here stand in contrast to the examples of Section \ref{sec:non-sigma-compact}. 

In light of Lemma \ref{lem:4.14} we have to ensure that for these examples, whenever a geodesic has sublinear intersection function, it has to have bounded intersection function. We ensure this by making sure that the overlaps in relators are `large enough' while not violating the $C'(1/6)$--small-cancellation condition.

More precisely, we start with a certain set of relators and say that they are `level 1' relators. We then construct the rest of the relators inductively by level. During the construction we make sure that relators of level $i+1$ are concatenations of subwords of level $i$ relators. Doing this the right way, `long' subwords of a relator contain a substantial fraction of a lower level relator, which can be used to show that every geodesic ray is either not Morse or strongly contracting. 

\begin{defn}
We say that a word $w = s_1\ldots s_n$ is non-periodic if $\abs{\ov{w}} = 2n$. That is, none of the cyclic shifts of $w$ or their inverses are equal to $w$.
\end{defn}

\subsection{Construction}

Let $N\geq 28$, let $L$ be a multiple of $N$ and let $\mc S$ be a finite set of formal variables. Let $r_1^1, r_2^1$ and $r_3^1$ be distinct non-periodic words over $\mc S$ of length $L$ such that $\mc R_1 = \{r_1^1, r_2^1, r_3^1\}$ satisfies the $C'(1/(4N))$--small-cancellation condition. For example $\mc S = \{s_1, s_2, \ldots, s_{4N}\}$, and $r_i^1 = s_{Ni +1}s_{Ni+2}\ldots s_{Ni+N}$ for $1\leq i \leq 3$. We think of $\mc R_1$ as the set of level 1 relators.
\begin{rem}
    It is important that $r_1^1, r_2^1$ and $r_3^1$ are words over $\mc S$ instead of words over $\ov{\mc S}$. Because of this, they are reduced and cyclically reduced. The same holds for all words constructed in the remainder of this section.
\end{rem}

\textbf{Convention:} For the rest of this section, unless stated otherwise, $1\leq i, i', i'' \leq 3$ and $1\leq  l, l', l''\leq N$ are integers. Furthermore, we think of $i$, $i$ and $i''$ as integers modulo 3 and $l, l'$ and $l''$ as integers modulo $N$. So if $i=3$, then $i+1$ is equal to $1$.

Write $r_i^1 = y_{(i, 1)}^1y_{(i, 2)}^1\ldots y_{(i, N-1)}^1y_{(i,N)}^1$ for subwords $y_{(i,l)}^1$ of length $\abs{r_i^1}/N$. For $k\geq 1$ iteratively define 
\begin{align*}
    y_{(i, l)}^{k+1}=  \prod_{j=1}^{N}y_{(i , j)}^ky_{(i+1, l)}^k.
\end{align*}
In other words, if $a =y_{(i, l)}^k$ and $b_j=y_{(i+1, j)}^k$, then 
\begin{align*}
    y_{(i,l)}^{k+1} = \prod_{j=1}^{N}ab_j.
\end{align*}
For $k\geq 2$ define
\begin{align}
    r_{i}^k = \prod_{j=1}^{N} y_{(i, j)}^k,
\end{align}
and define $\mc R = \cup_{i=1}^3\cup_{k=1}^\infty r_i^k$. We think of relators $r_i^k$ as relators of level $k$. With this definition we assure the following; while a particular word $y^k_{(i, l)}$ appears several times in the definition of the $y^{k+1}$, any word $y_{(i, l)}^ky_{(i', l')}^k$ appears at most once in the definition of the $y^{k+1}$.

This property is a key step in showing that $\mc R$ satisfies the $C'(1/6)$--small-cancellation condition, while ensuring that if a geodesic has a large common subword with relator $r_*^k$, then it also has a relatively large common subword with a relator $r_*^{k'}$ for some $k'<k$.

\begin{lem}\label{lemma:structure_of_ys}
Let $1\leq k'< k$. We can write $y_{(i,l)}^{k}$ as the product
\begin{align*}
    y_{(i,l)}^{k} = \prod_{j=1}^m y_{(i_j, l_j)}^{k'},
\end{align*}
for some integer $m\geq 2$ such that:
\begin{enumerate}  
    \item $i_1\neq i_2$ and $i_{m-1}\neq i_m$,
    \item if $i_{l} = i_{l+1}$, then $i_{l-1}\neq i_l$ and $i_{l+2}\neq i_{l+1}$.
\end{enumerate}
\end{lem}

\begin{proof}
    If $k' =k-1$, this follows from the definition of $y_{(i, l)}^{k}$. Otherwise it follows by induction on $k$.
\end{proof}

\begin{lem}\label{lemma:structure_of_ri}
Let $1\leq k' < k$ we can write $r_i^{k}$ as the concatenation
\begin{align*}
    r_i^{k} = \prod_{j=1}^m y_{(i_j, l_j)}^{k'},
\end{align*}
for some integer $m\geq 2$ such that 
\begin{enumerate}
    \item $i_1\neq i_2$ and $i_{m-1}\neq i_m$.
    \item if $i_{l} = i_{l+1}$, then $i_{l-1}\neq i_l$ and $i_{l+2}\neq i_{l+1}$.
\end{enumerate}
\end{lem}
\begin{proof}
    This is an immediate consequence of the definition of $r_i^{k}$ and Lemma \ref{lemma:structure_of_ys}.
\end{proof}

The following lemma explores how a word $y_{(i, l)}^k$ can be a subword of a product of terms $y_{(i', l')}^k$ and is a key ingredient in showing that $\mc R$ satisfies the $C'(1/6)$--small-cancellation condition.

\begin{lem}\label{lemma:star}
Let $k\geq 1$ and let $w = \prod_{j=1}^m y_{(i_j ,l_j)}^k$ be a word. If $y_{(i, l)}^k$ is a subword of $w$, then there exists an index $j$ such that $(i_j, l_j) = (i, l)$. Furthermore, if $w$ can be written as $uy_{(i, l)}^kv$, then $u = \prod_{j=1}^{n-1}y_{(i_j, l_j)}^k$, for some $1\leq n \leq m$ with $(i_n, l_n) = (i, l)$.
\end{lem}

\begin{proof}
    If $a = y_{(i, l)}^1$ is a subword of $w$, then either $w = uav$, there exists $1\leq j<m$ such that $y_{(i_j, l_j)}^ky_{(i_{j+1}, l_{j+1})}^k = u'pqv'$ for some words $u', v', p, q$ with $pq =a$. It suffices to show that either $u'$ is empty and $(i, l) = (i_j, l_j)$ or $v'$ is empty and $(i, l) = (i_{j+1}, l_{j+1})$. We first show this for $k=1$ and then by induction on $k$. 
    
    If $k=1$ and the statement above does not hold, then $p$ and $q$ are both pieces. Hence $L/N = \abs{a} = \abs{p}+\abs{q} < L/(2N)$, a contradiction.

    Assume the statement holds for $k-1$. Define $b = y_{(i_j, l_j)}^k$ and $c =y_{(i_{j+1}, l_{j+1})}^k$. Recalling their definitions, we can write $a, b$ and $ c$ as product of words of the form $y_{(i', l')}^{k-1}$. The term $y_{(i, l)}^{k-1}$ appears $N$ times in the product of $a$, and appears at most once in the product of $y_{(i', l')}^k$ unless $(i', l') = (i, l)$. Applying the statement for $k-1$ and $bc$ we get that $b$ or $c$ contains $y_{(i, l)}^{k-1}$ multiple times, implying that $(i_j, l_j) = (i, l)$ or $(i_{j+1}, l_{j+1}) = (i, l)$. Moreover, the terms $y_{(i+1, 1)}^{k-1}$ and $y_{(i+1, N)}^{k-1}$ appear exactly once (and hence at a unique spot) when writing $y_{(i, l)}^k$ as a product, so we actually have either $(i_j, l_j) = (i, l)$ and $u'$ is the empty word or $(i_{j+1}, l_{j+1}) = (i, l)$ and $v'$ is the empty word, which shows that the statement holds for $k$.
\end{proof}

Instead of showing that $\mc R$ satisfies the $C'(1/6)$--small-cancellation condition we show a slightly stronger result, which we need for the next section. 

\begin{lem}
The set of relators $\mc R =$ satisfies the $C'(1/7)$--small-cancellation condition.  
\end{lem}

\begin{proof}
    It suffices to show that the common prefix $p$ of a pair of distinct relators $r, r'\in \ov {\mc R}$ has length less than $\abs{r}/7$ and $\abs{r'}/7$. Since all the $r_i^k$ are words over $\mc S$, we may assume that $r$ and $r'$ are cyclic shifts of relators $r_i^k$ and $r_{i'}^{k'}$ for some $k'\geq k$.
    
    \textbf{Case 1: } $k = k'$. Assume by contradiction that $\abs{p} \geq \abs{r}/7$. Recall that $r_i^k = y_{(i, 1)}^ky_{(i, 2)}^{k}\ldots y_{(i, N)}^k$. Since $N\geq 28$, there exists an integer $l$ such that $y_{(i, l)}^k$ is a subword of $\abs{p}$. Applying Lemma \ref{lemma:star} we get a contradiction. Hence $\abs{p}< \abs{r}/7 = \abs{r'}/7$.
    
    \textbf{Case 2: } $k' > k$. Assume by contradiction that $\abs{p}\geq \abs{r}/7$. Since $N\geq 28$, there exists and integer $l$ such that $y_{(i, l)}^{k}y_{(i, l+1)}^{k}y_{(i, l+2)}^k$ is a subword of $p$. Write $r_{i'}^{k'}$ as in Lemma \ref{lemma:structure_of_ri}. By Lemma \ref{lemma:star}, there exists an integer $l'$ such that $i_{l'} = i_{l'+1} = i_{l'+2}$, a contradiction to (2) of Lemma \ref{lemma:structure_of_ri}. Hence $\abs{p} < \abs{r}/7<\abs{r'}/7$.
\end{proof}

\begin{prop}\label{prop:sigma_compact}
    Every Morse geodesic ray in $X = \cay$ is strongly contracting.
\end{prop}

\begin{proof}
    Assume that $\gamma$ is a Morse geodesic ray which is not strongly contracting. In light of Lemma \ref{lem:4.14}, its intersection function
    \[
\rho(x) = \max_{\abs{r}\leq x}\left\{\abs{w}\Big\vert \text{$w$ is a subword of $r$ and $\gamma$}\right\},
\]
    is sublinear but unbounded. Let $k\geq 1$. Since $\rho$ is unbounded, there exists a relator $r\in \ov{ \mc R}$ and a common subword $w$ of $r$ and $\gamma$ of length at least $\abs{r_1^k}$. In light of Lemma \ref{lemma:structure_of_ri}, there exist $(i, l)$ such that $y_{(i, l)}^k$ is a subword of $w$ or $w^{-1}$, implying that either $r_i^k$ or $(r_i^k)^{-1}$ have a common subword with $\gamma$ of length at least $\abs{r_i^k}/N$. Consequently, $\rho(\abs{r_i^k})\geq \abs{r_i^k}/N$. Since this holds for all $k$, $\rho$ cannot be sublinear, a contradiction.
\end{proof}

\section{Sigma compact but not all Morse rays are strongly contracting}\label{sec:intermidate}

In this section we modify the group constructed in the previous section. The modified group is a $C'(1/6)$--group which has $\sigma$-compact Morse boundary but not all Morse rays in its Cayley graph are strongly contracting. This rounds out the set of examples we give and shows that there are $C'(1/6)$--groups whose Morse boundary is $\sigma$-compact but not all Morse rays are strongly contracting.

\smallskip
We use the notation from the previous section. Let ${\mc S}' = \mc S\cup\{a\}$ for some formal variable $a\not\in \mc S$. For $1\leq i \leq 2$ define 
$\tilde {r}_i^k = r_i^k$ and for $i=3$ define
\begin{align}\label{eq5}
    \tilde{r}_3^k = \begin{cases}
    r_3^ka^k&\text{if  $k\leq \abs{r_3^k/42}$},\\
    r_3^k &\text{otherwise.}
    \end{cases}
\end{align}
Define ${\mc R}' = \cup_{k\geq 1}\cup_{i=1}^3\{\tilde{r}_i^k\}$.

\begin{lem}
     The set $\mc R'$ satisfies the  $C'(1/6)$--small-cancellation condition.
\end{lem}
\begin{proof}
    Let $p$ be a piece of a relator $r\in \ov{\mc R}'$ which is a cyclic shift of a relator $\tilde{r}_i^k$. Then either $p = q a^j$ or $p = a^j q$ for some $0\leq j \leq k$ and subword $q$ which, in $\mc R$, is a piece of $r_i^k$. By \eqref{eq5}, $j\leq \abs{r}/42 $ and since $\mc R$ satisfies the $C'(1/7)$--small-cancellation condition $\abs{q}< \abs{r_i^k}/7$. Hence $\abs{p}< \abs{r}/6$. 
\end{proof}

\begin{prop}
     The group $G = \langle \mc S'|\mc R'\rangle$ has sigma compact Morse boundary but not every Morse ray in $X = \mathrm{Cay}(G, \mc S')$ is strongly contracting.
\end{prop}

\begin{proof}
    We first show that there exists a geodesic Morse ray in $X$ which is not strongly contracting. Define the sublinear function $\rho_0(x) = \log_{N}(x)+2$. Let $\gamma$ be the ray in $X$ starting at the identity, where every edge on $\gamma$ is labelled by $a$. Let $w$ be a common subword of a relator $r\in \ov{\tilde{r}_i^k}$. By \eqref{eq5} we have that $\abs{w}\leq \abs{r}/42$. Hence $\gamma$ is a geodesic by Lemma \ref{lemma:geodesic_condition}. Further, if $i\neq 3$, then $w$ is the empty word. Otherwise $\abs{w}\leq k$. By construction, $\abs{r_i^k} = (2N)^{k-1}\abs{r_i^1}$. Thus, $\abs{w}\leq \rho_0(\abs{r})$. This shows that the intersection function $\rho$ of $\gamma$ satisfies $\rho\leq \rho_0$, implying that $\gamma$ is Morse by Lemma \ref{lem:4.14}. Further, the intersection function $\rho$ is unbounded and hence $\gamma$ is not strongly contracting by Lemma \ref{lem:4.14}.
    
    Next we show that $\mb X$ is $\sigma$-compact. Morse precisely, for every integer $j\geq 1$ define $\rho_j(x) = \rho_0(x)+2j$. We show that $\rho_j(x)$ is a contraction exhaustion of $\mb X$. Let $\gamma$ be a Morse geodesic ray and let $\rho$ be its intersection function. The proof of Proposition \ref{prop:sigma_compact} shows that there exists some bound $D$ such that every subword of $\gamma$ which is also a subword of a relator $r\in \mc R$ has length at most $D$. Thus for any common subword $w$ of $\gamma$ and a relator $r\in \ov{\tilde{r}_i^k}$ we have that $\abs{w}\leq 2D + k\leq \rho_D(\abs{r})$. Hence $\rho\leq \rho_D$, which concludes the proof. 
\end{proof}

\bibliography{mybib}

\begin{thebibliography}{ACGH19}

\bibitem[ACGH17]{arzhantseva2017characterizations}
Goulnara~N Arzhantseva, Christopher~H Cashen, Dominik Gruber, and David Hume.
\newblock Characterizations of morse quasi-geodesics via superlinear divergence
  and sublinear contraction.
\newblock {\em Documenta Mathematica}, 22:1193--1224, 2017.

\bibitem[ACGH19]{arzhantseva2019negative}
Goulnara~N Arzhantseva, Christopher~H Cashen, Dominik Gruber, and David Hume.
\newblock Negative curvature in graphical small cancellation groups.
\newblock {\em Groups, Geometry, and Dynamics}, 13(2):579--632, 2019.

\bibitem[CCS19]{charney2019complete}
Ruth Charney, Matthew Cordes, and Alessandro Sisto.
\newblock Complete topological descriptions of certain morse boundaries.
\newblock {\em to appear in Groups Geom. Dyn.}, 2019.

\bibitem[CD19]{CD:stable}
Matthew Cordes and Matthew~Gentry Durham.
\newblock Boundary convex cocompactness and stability of subgroups of finitely
  generated groups.
\newblock {\em Int. Math. Res. Not. IMRN}, (6):1699--1724, 2019.

\bibitem[Cor17a]{C:Morse}
Matthew Cordes.
\newblock Morse boundaries of proper geodesic metric spaces.
\newblock {\em Groups Geom. Dyn.}, 11(4):1281--1306, 2017.

\bibitem[Cor17b]{C:survey}
Matthew Cordes.
\newblock A survey on morse boundaries \& stability.
\newblock {\em arXiv preprint arXiv:1704.07598}, 2017.

\bibitem[CS15]{CS:contracting}
Ruth Charney and Harold Sultan.
\newblock Contracting boundaries of {$\rm CAT(0)$} spaces.
\newblock {\em J. Topol.}, 8(1):93--117, 2015.

\bibitem[Gru15]{gruber2015groups}
Dominik Gruber.
\newblock Groups with graphical c(6) and c(7) small cancellation presentations.
\newblock {\em Transactions of the American Mathematical Society},
  367(3):2051--2078, 2015.

\bibitem[GS18]{gruber2018infinitely}
Dominik Gruber and Alessandro Sisto.
\newblock Infinitely presented graphical small cancellation groups are
  acylindrically hyperbolic.
\newblock In {\em Annales de l'Institut Fourier}, volume~68, pages 2501--2552,
  2018.

\bibitem[HHP20]{HHP:injective}
Thomas Haettel, Nima Hoda, and Harry Petyt.
\newblock Coarse injectivity, hierarchical hyperbolicity, and
  semihyperbolicity.
\newblock {\em arXiv preprint arXiv:2009.14053}, 2020.

\bibitem[LSLS77]{lyndon1977combinatorial}
Roger~C Lyndon, Paul~E Schupp, RC~Lyndon, and PE~Schupp.
\newblock {\em Combinatorial group theory}, volume 188.
\newblock Springer, 1977.

\bibitem[Mur19]{M:CAT0}
Devin Murray.
\newblock Topology and dynamics of the contracting boundary of cocompact cat
  (0) spaces.
\newblock {\em Pacific Journal of Mathematics}, 299(1):89--116, 2019.

\bibitem[Str90]{strebel1990small}
Ralph Strebel.
\newblock Small cancellation groups.
\newblock In {\em Sur les groupes hyperboliques d’apr{\`e}s Mikhael Gromov},
  pages 227--273. Birkh{\"a}user Basel, 1990.

\bibitem[Sul14]{S:CAT0}
Harold Sultan.
\newblock Hyperbolic quasi-geodesics in cat (0) spaces.
\newblock {\em Geometriae Dedicata}, 169(1):209--224, 2014.

\bibitem[SZ22]{SZ:injective}
Alessandro Sisto and Abdul Zalloum.
\newblock Morse subsets of injective spaces are strongly contracting.
\newblock {\em arXiv preprint arXiv:2208.13859}, 2022.

\bibitem[Zbi22]{Z:manifold}
Stefanie Zbinden.
\newblock Morse boundaries of 3-manifold groups.
\newblock {\em arXiv preprint arXiv:2212.08368}, 2022.

\end{thebibliography}
\bibliographystyle{alpha}

\end{document}